\documentclass{article}
\usepackage[affil-it]{authblk}
\usepackage[utf8]{inputenc}
\usepackage{amsmath, amsfonts, amsthm, amssymb, latexsym}
\usepackage{geometry}
\usepackage{enumerate}
\usepackage{multirow}
\usepackage{rotating}
\usepackage{graphicx,color}
\usepackage{float}
\usepackage{mathrsfs}

\newtheorem{thm}{Theorem}[section]

\newtheorem{prop}[thm]{Proposition}

\newtheorem*{remark}{Remark}

\title{On inverse problems arising in fractional elasticity}
\author{Li Li}
\affil{Institute for Pure and Applied Mathematics, University of California,\\
Los Angeles, CA 90095, USA}
\date{}

\begin{document}
\maketitle

\noindent \textbf{ABSTRACT.}\, We first formulate an inverse problem for a linear fractional Lam\'e system. We determine the Lam\'e parameters from exterior partial measurements of the Dirichlet-to-Neumann map. We further study an inverse obstacle problem as well as an inverse problem for a nonlinear fractional Lam\'e system. Our arguments are based on the unique continuation property for the fractional operator as well as the associated Runge approximation property.

\section{Introduction}
The classical Lam\'e operator $L_{\lambda, \mu}$ for a three-dimensional isotropic elastic body is given by
$$(L_{\lambda, \mu}u)_i:= \sum^3_{j=1}\partial_i(\lambda u_{j, j})
+ \sum^3_{j=1}\partial_j(\mu (u_{i, j}+ u_{j, i})),\quad (1\leq i\leq 3)$$
where $u_i$ denotes the $i$-th component of the vector-valued displacement function $u$ and $u_{i,j}:= \partial_j u_i$. 
The associated inverse problem has been studied in \cite{eskin2002inverse, nakamura1994global} where the authors considered the Dirichlet problem
$$L_{\lambda, \mu}u= 0\,\,\,\, \text{in}\,\,\Omega,\qquad
u= g\,\,\,\, \text{on}\,\,\partial\Omega.$$
They determined the variable Lam\'e parameters $\lambda, \mu$ from the Dirichlet-to-Neumann (displacement-to-traction) map 
$$(\Lambda g)_i:= \sum^3_{j=1}(\lambda u_{j, j})\nu_i
+ \sum^3_{j=1}\mu (u_{i, j}+ u_{j, i})\nu_j,\quad (1\leq i\leq 3)$$
where $\nu$ is the unit outer normal to $\partial\Omega$ under certain assumptions on $\lambda, \mu$. 

In this paper, we study a fractional analogue of $L_{\lambda, \mu}$ and its associated inverse problems. 

First we recall that the classical elastic model is based on the constitutive relation
$$\sigma_{ij}= \sum_{k,l= 1}^3C_{ijkl}e_{kl}$$
where the fourth-order elastic stiffness tensor is 
$$C_{ijkl}:= \lambda\delta_{ij}\delta_{kl}+\mu(\delta_{ik}\delta_{jl}+ \delta_{il}\delta_{jk}),$$
the linearized strain tensor is
$$e_{ij}:= \frac{1}{2}(u_{i,j}+u_{j,i})$$
and $\sigma_{ij}$ denotes the stress tensor. 
Here $\delta_{ij}$ is the standard Kronecker delta.

Recently, the theory of nonlocal elasticity has attracted much attention. The integral linear constitutive relation
$$\sigma_{ij}= \sum_{k,l= 1}^3C_{ijkl}(K*e_{kl})$$
has been introduced to describe complex materials characterized by nonlocality. 
Then the fractional Taylor series approximation for the Fourier transform of the interaction kernel $K$, which is given by 
$$\hat{K}(|\xi|)\approx \hat{K}(0)+ c_s|\xi|^{2s},\quad (0<s<1)$$
leads to the definition of the fractional Lam\'e operator
\begin{equation}\label{fracLame} 
\mathcal{L}:= L_{\lambda, \mu}+ (-\Delta)^sL_{\lambda_0, \mu_0}.
\end{equation}
See \cite{tarasov2019fractional}
and the references there for more background information.
In \cite{tarasov2019fractional}, $\lambda, \mu$ are constants proportional to $\lambda_0, \mu_0$ but in this paper we allow $\lambda, \mu$ to be variable functions.

We consider the exterior Dirichlet problem
\begin{equation}\label{fracLameext}
\mathcal{L}u= 0\,\,\,\, \text{in}\,\,\Omega,\qquad
u= g\,\,\,\, \text{in}\,\,\Omega_e  
\end{equation}
where $\Omega$ is a bounded Lipschitz domain and $\Omega_e:= \mathbb{R}^3\setminus \bar{\Omega}$.
Under appropriate assumptions on $\lambda_0, \mu_0, \lambda, \mu$, we can show its well-posedness
so we will be able to define the associated Dirichlet-to-Neumann map $\Lambda$, which is formally given by
\begin{equation}\label{DN}
\Lambda g:= (-\Delta)^sL_{\lambda_0, \mu_0}u_g|_{\Omega_e}.
\end{equation}
Our goal here is to determine both $\lambda$ and $\mu$ from exterior partial measurements of $\Lambda$. 

We remark that our problem can be viewed as a variant of the fractional Calder\'on problem first introduced in \cite{ghosh2020calderon} where the authors considered the
exterior Dirichlet problem
$$((-\Delta)^s+ q)u= 0\,\,\, \text{in}\,\,\Omega,\qquad
u= g\,\,\,\, \text{in}\,\,\Omega_e$$ and they proved the fundamental uniqueness theorem that the potential $q$ in $\Omega$ can be determined from exterior partial measurements of the map 
$$\Lambda_q: g\to (-\Delta)^su_g|_{\Omega_e}.$$

It has been shown that the knowledge of $\Lambda_q g$ is equivalent to 
the knowledge of the nonlocal Neumann derivative of $u_g$ (see \cite{ghosh2020calderon} for more details). Hence our problem can also be viewed as a nonlocal analogue of
the inverse problem for the classical Lam\'e system.

We mention that inverse problems for fractional operators have been extensively studied so far. See \cite{ruland2020fractional} for low regularity
and stability results for the fractional Calder\'on problem. See \cite{ghosh2020uniqueness} for
reconstruction and single measurement results for the fractional Calder\'on problem. See \cite{ghosh2017calderon} for inverse problems for variable
coefficients fractional elliptic operators. See \cite{bhattacharyya2021inverse, covi2020higher} for inverse problems for
fractional Schr\"odinger operators with local and non-local perturbations.
See \cite{covi2020inverse, lai2021inverse, li2020determining} for inverse problems for fractional magnetic operators. See \cite{lai2020calderon, li2021fractional, li2021inverse} for inverse problems for fractional parabolic operators.

The following theorem is our first main result in this paper.
\begin{thm}
Let $0< s< 1$. Let $\mu_0> 0$ and $\lambda_0+ \mu_0\geq 0$. 
Let $0\leq\lambda^{(j)}, \mu^{(j)}\in C^1(\bar\Omega)$ and
let $W_j\subset \Omega_e$ be nonempty and open ($j= 1,2$). Suppose 
$$\Lambda^{(1)}g|_{W_2}= \Lambda^{(2)}g|_{W_2}$$
for all $g\in C^\infty_c(W_1)$. Then 
$\lambda^{(1)}= \lambda^{(2)}, \mu^{(1)}= \mu^{(2)}$ in $\Omega$.
\end{thm}

We remark that our problem provides an example which suggests that the inverse problem for the fractional operator is more manageable than its classical counterpart. 

Recall that to solve the classical inverse problem, we first reduce the Lam\'e
system to a first order system perturbation of the
Laplacian. Then we construct complex geometrical optics (CGO) solutions and apply the integral identity to obtain the uniqueness of $\lambda, \mu$. For some technical reasons, the uniqueness result is only proved provided that $\mu$ is close to a constant.
The full classical problem remains open.
See \cite{eskin2002inverse, nakamura2003global} for details. 

Here such a priori knowledge of $\mu$ is not required for solving the fractional problem.
Instead of constructing CGO solutions, we will use the unique continuation property and the Runge approximation property associated with our fractional operator to prove the strong uniqueness result. This scheme was first introduced in \cite{ghosh2020calderon} for solving the fractional Calder\'on problem. 

We further study an inverse obstacle problem associated with our fractional operator.

We consider the following obstacle problem
\begin{equation}\label{obeq}
\mathcal{L}u= 0\,\, \text{in}\,\,\Omega\setminus\bar{D},\quad
u= 0\,\, \text{in}\,\,D,\quad
u= g\,\, \text{in}\,\,\Omega_e
\end{equation}
where $D\subset \Omega$ is a nonempty open set satisfying that 
$\Omega\setminus \bar{D}$ is a bounded Lipschitz domain. 

As we did for the exterior problem (\ref{fracLameext}), we can similarly show 
the well-posedness of (\ref{obeq}) and define the Dirichlet-to-Neumann map $\Lambda_D$ by 
\begin{equation}\label{DNforD}
\Lambda_D g:= (-\Delta)^sL_{\lambda_0, \mu_0}u_g|_{\Omega_e}.
\end{equation}
Our next goal is to determine $D, \lambda, \mu$ from the knowledge of $\Lambda_D$. 

The following theorem is our second main result.

\begin{thm}
Let $0< s< 1$. Let $\mu_0> 0$ and $\lambda_0+ \mu_0\geq 0$. 
Let $D_j\subset \Omega$ be nonempty and open s.t.
$\Omega\setminus \bar{D_j}$ is a bounded Lipschitz domain, $0\leq\lambda^{(j)}, \mu^{(j)}\in C^1(\bar\Omega\setminus D_j)$ and
let $W_j\subset \Omega_e$ be nonempty and open ($j= 1,2$). Suppose 
$$\Lambda^{(1)}_{D_1} g|_{W_2}= \Lambda^{(2)}_{D_2} g|_{W_2}$$
for a nonzero $g\in C^\infty_c(W_1)$. Then $D_1= D_2=: D$.
Further assume the identity holds for all $g\in C^\infty_c(W_1)$.
Then $\lambda^{(1)}=\lambda^{(2)}$ and $\mu^{(1)}=\mu^{(2)}$ in $\Omega\setminus\bar{D}$.
\end{thm}

We also study an inverse problem for a nonlinear fractional Lam\'e system.

We consider the following nonlinear exterior problem
\begin{equation}\label{NfracLameext}
\mathcal{L}u+ \mathcal{N}u= 0\,\,\,\, \text{in}\,\,\Omega,\qquad
u= g\,\,\,\, \text{in}\,\,\Omega_e  
\end{equation}
where the nonlinear operator $\mathcal{N}$ is given by
\begin{equation}\label{Nsum}
(\mathcal{N}u)_i:= \sum^3_{j=1}\partial_jN_{ij}u,\quad (1\leq i\leq 3)
\end{equation}
$$N_{ij}u:= \frac{\lambda+\mathcal{B}}{2}\sum_{m,n}u^2_{m,n}\delta_{ij}+
\mathcal{C}(\sum_m u_{m,m})^2\delta_{ij}
+ \frac{\mathcal{B}}{2}\sum_{m,n}u_{m,n}u_{n,m}\delta_{ij}$$
$$+\mathcal{B}\sum_m u_{m,m}u_{j,i}
+ \frac{\mathcal{A}}{4}\sum_m u_{j,m}u_{m,i}+ (\lambda+ \mathcal{B})\sum_m u_{m,m}u_{i,j}$$
$$+(\mu+ \frac{\mathcal{A}}{4})\sum_m(u_{m,i}u_{m,j}+ u_{i,m}u_{j,m}+ u_{i,m}u_{m,j}).$$
This nonlinearity comes from the higher order expansion of the energy density as well as the nonlinear term in the strain tensor
$$e_{ij}:= \frac{1}{2}(u_{i,j}+u_{j,i}+\sum_m u_{m,i}u_{m,j}).$$
We remark that our $\mathcal{N}$ is the static version of the nonlinearity considered in \cite{de2020nonlinear, uhlmann2021inverse}, where the inverse problem for the associated nonlinear elastic wave equation was studied. See the references there for more background information on $\mathcal{N}$. 

Under certain assumptions, we can show the well-posedness of (\ref{NfracLameext}) for small $g\in C^\infty_c(\Omega_e)$
and then for such $g$ we can define the associated Dirichlet-to-Neumann map $\Lambda_N$ formally given by
\begin{equation}\label{DNforN}
\Lambda_N g:= (-\Delta)^sL_{\lambda_0, \mu_0}u_g|_{\Omega_e}.
\end{equation}
Our last goal is to determine $\lambda, \mu, \mathcal{A}, \mathcal{C}$ from the knowledge of $\Lambda_N$. 

The following theorem is our third main result.
\begin{thm}
Let $\frac{1}{2}\leq s< 1$. Let $\mu_0> 0$ and $\lambda_0+ \mu_0\geq 0$ and let $\mathcal{B}\in C^1(\bar\Omega)$. 
Let $0\leq\lambda^{(j)}, \mu^{(j)}\in C^1(\bar\Omega)$,
let $\mathcal{A}^{(j)}, \mathcal{C}^{(j)}\in C^1(\bar\Omega)$ and
let $W_j\subset \Omega_e$ be nonempty and open ($j= 1,2$). Suppose 
$$\Lambda^{(1)}_Ng|_{W_2}= \Lambda^{(2)}_Ng|_{W_2}$$
for small $g\in C^\infty_c(W_1)$. Then 
$\lambda^{(1)}= \lambda^{(2)}, \mu^{(1)}= \mu^{(2)}, \mathcal{A}^{(1)}= \mathcal{A}^{(2)}, \mathcal{C}^{(1)}= \mathcal{C}^{(2)}$ in $\Omega$.
\end{thm}

Note that here we only claim that $\mathcal{A}, \mathcal{C}$ can be determined for a fixed $\mathcal{B}$. The question whether we can simultaneously determine $\mathcal{A},\mathcal{B},\mathcal{C}$ is still open.

The rest of this paper is organized in the following way. In Section 2, we summarize the background knowledge. In Section 3, we show the well-posedness of the linear exterior problem; We prove the unique continuation property and the Runge approximation property associated with our fractional operator; Then we prove Theorem 1.1 and Theorem 1.2. In Section 4, we show the well-posedness of the nonlinear exterior problem for small exterior data; We combine linearization arguments with the Runge approximation property to prove Theorem 1.3.
\medskip

\noindent \textbf{Acknowledgements.} The author would like to thank Professor Gunther Uhlmann for suggesting the problem and for helpful discussions.

\section{Preliminaries}

Throughout this paper we use the following notations.
\begin{itemize}
\item We fix the space dimension $n=3$ and the fractional power $0< s< 1$.

\item $x= (x_1, x_2, x_3)$ denotes the spatial variable.

\item For vector-valued function $u$, $u_i$ denotes the $i$-th component of $u$ and $u_{i,j}:= \partial_j u_i$. 

\item $\Omega$ denotes a bounded Lipschitz domain and
$\Omega_e:= \mathbb{R}^n\setminus\bar{\Omega}$.

\item $\delta_{ij}$ denotes the standard Kronecker delta.

\item $\langle\cdot, \cdot\rangle$ denotes the distributional pairing so formally, $\langle f, g\rangle= \int fg$.
\end{itemize}
Throughout this paper we refer all function spaces to real-valued function spaces. For convenience, we use the same notation for the scalar-valued function space and the vector-valued one. For instance, $C^\infty_c(\Omega)$
can be either $C^\infty_c(\Omega; \mathbb{R})$ or $C^\infty_c(\Omega; \mathbb{R}^3)$.

\subsection{Sobolev spaces}
For $r\in \mathbb{R}$, we have the Sobolev space
$$H^r(\mathbb{R}^n):= \{f\in \mathcal{S}'(\mathbb{R}^n):
\int_{\mathbb{R}^n} (1+\vert \xi\vert^2)^r\vert \mathcal{F}f(\xi)\vert^2d\xi<\infty\}$$
where $\mathcal{F}$ is the Fourier transform and $\mathcal{S}'(\mathbb{R}^n)$ is the space of temperate distributions.

We have the natural identification
$$H^{-r}(\mathbb{R}^n)= H^r(\mathbb{R}^n)^*.$$
Let $U$ be an open set in $\mathbb{R}^n$. Let $F$ be a closed set in $\mathbb{R}^n$. Then
$$H^r(U):= \{u|_U: u\in H^r(\mathbb{R}^n)\},\qquad 
H^r_F(\mathbb{R}^n):= 
\{u\in H^r(\mathbb{R}^n): \mathrm{supp}\,u\subset F\},$$
$$\tilde{H}^r(U):= 
\mathrm{the\,\,closure\,\,of}\,\, C^\infty_c(U)\,\,\mathrm{in}\,\, H^r(\mathbb{R}^n).$$
Since $\Omega$ is a bounded Lipschitz domain, we also have the identifications
$$\tilde{H}^r(\Omega)= H^r_{\bar{\Omega}}(\mathbb{R}^n),\quad 
H^{-r}(\Omega)= \tilde{H}^r(\Omega)^*.$$

Let $0< s< 1$. It is well known that we have the following continuous embedding
$$\tilde{H}^s(\Omega)\hookrightarrow L^{\frac{2n}{n-2s}}(\Omega).$$
See for instance, Section 1.5 in \cite{bisci2016variational}. It has also been proved that we have 
the continuous embedding
$$L^{\frac{n}{2s}}(\Omega)\hookrightarrow M(H^s\to H^{-s})$$
where $M(H^s\to H^{-s})$ is the space of pointwise multipliers from $H^s(\mathbb{R}^n)$ to $H^{-s}(\mathbb{R}^n)$ equipped with the norm
$$||f||_{s, -s}= \sup\{|\langle f, \phi\psi\rangle|:\,
\phi,\psi\in C^\infty_c(\Omega),\,\,||\phi||_{H^s(\mathbb{R}^n)}
=||\psi||_{H^s(\mathbb{R}^n)} = 1\}.$$
See for instance, Section 2 in \cite{covi2020higher, ruland2020fractional}.

\subsection{Fractional Laplacian}
Let $0< s< 1$. The fractional Laplacian $(-\Delta)^s$ is formally given 
by the pointwise definition 
$$(-\Delta)^su(x):= 
c_{n,s}\lim_{\epsilon\to 0^+}
\int_{\mathbb{R}^n\setminus B_\epsilon(x)}\frac{u(x)-u(y)}{|x-y|^{n+2s}}\,dy$$
as well as the the equivalent Fourier transform definition
$$(-\Delta)^su(x):= \mathcal{F}^{-1}(|\xi|^{2s}\mathcal{F}u(\xi))(x).$$

It is well-known that one of the equivalent forms of the $H^{s}$-norm 
is given by 
$$||f||^2_{H^{s}}:= ||f||^2_{L^2}+ \int_{\mathbb{R}^n}\int_{\mathbb{R}^n}
\frac{|f(x)- f(y)|^2}{|x-y|^{n+ 2s}},\quad f\in H^{s}(\mathbb{R}^n)$$
and we have the following bilinear form formula
$$\langle (-\Delta)^su, v\rangle=  c'_{n,s}\int_{\mathbb{R}^n}\int_{\mathbb{R}^n}\frac
{(u(x)-u(y))(v(x)-v(y))}{|x-y|^{n+2s}}\,dxdy,\quad 
u, v\in H^{s}(\mathbb{R}^n).$$
It is also well-known that one of the equivalent forms of the $H^{1+s}$-norm is given by 
$$||f||^2_{H^{1+s}}:= ||f||^2_{H^1}+ \sum_j\int_{\mathbb{R}^n}\int_{\mathbb{R}^n}
\frac{|\partial_j f(x)- \partial_j f(y)|^2}{|x-y|^{n+ 2s}},\quad f\in H^{1+s}(\mathbb{R}^n).$$
By the classical and fractional Poincar\'e inequalities, we have the following norm equivalence
$$||f||^2_{H^{1+s}}\sim \sum_j(||\partial_j f||^2_{L^2}+ 
\langle (-\Delta)^s \partial_j f, \partial_j f\rangle)
\sim \sum_j\langle (-\Delta)^s \partial_j f, \partial_j f\rangle),\quad
f\in\tilde{H}^{1+s}(\Omega).$$
 
The following unique continuation property of $(-\Delta)^s$ was first proved in \cite{ghosh2020calderon}.
\begin{prop}\label{UCP}
Suppose $u\in H^r(\mathbb{R}^n)$ for some $r\in \mathbb{R}$. Let $W\subset \mathbb{R}^n$ be open and non-empty. If $$(-\Delta)^su= u= 0\quad\mathrm{in}\,\,W,$$
then $u= 0$ in $\mathbb{R}^n$.
\end{prop}

\section{Linear fractional elasticity}

\subsection{Well-posedness}
We first study the equation 
\begin{equation}\label{LameOmega}
\mathcal{L}u= f\,\,\,\,\text{in}\,\,\Omega.
\end{equation}
The bilinear form associated with $-\mathcal{L}$ (see (\ref{fracLame}) in Section 1 for its definition) is
$$B[u, v]= \langle \sum_j (-\Delta)^s\lambda_0 u_{j, j}, \sum_j v_{j,j}\rangle
+ \int_{\Omega}\lambda (\sum_j u_{j,j})(\sum_j v_{j,j})$$
\begin{equation}\label{Blf}
+ \frac{1}{2}\sum_{i,j}(\langle (-\Delta)^s\mu_0 (u_{i, j}+ u_{j, i}), 
v_{i, j}+ v_{j, i}\rangle+ \int_\Omega\mu(u_{i, j}+ u_{j, i})(v_{i, j}+ v_{j, i})).
\end{equation}

Assume $\lambda, \mu\in L^\infty(\Omega)$. Then it is clear that $B$ is bounded over $H^{1+s}(\mathbb{R}^3)\times H^{1+s}(\mathbb{R}^3)$ since
the bilinear form $\langle (-\Delta)^s\phi, \psi\rangle$ is bounded over $H^{s}(\mathbb{R}^3)\times H^{s}(\mathbb{R}^3)$.
Now we show that $B$ is coercive over 
$\tilde{H}^{1+s}(\Omega)\times \tilde{H}^{1+s}(\Omega)$
if we further assume
$\lambda, \mu\geq 0$, $\mu_0> 0$ and $\lambda_0+ \mu_0\geq 0$.

In fact, we note that
$$\langle (-\Delta)^s u_{i, j}, u_{j, i}\rangle=
\langle (-\Delta)^s u_{i, i}, u_{j, j}\rangle$$
for $u\in \tilde{H}^{1+s}(\Omega)$ so we have
$$\langle \sum_j (-\Delta)^s\lambda_0 u_{j, j}, \sum_j u_{j,j}\rangle
+ \frac{1}{2}\sum_{i,j}\langle (-\Delta)^s\mu_0 (u_{i, j}+ u_{j, i}), 
u_{i, j}+ u_{j, i}\rangle$$
$$= \langle \sum_j (-\Delta)^s\lambda_0 u_{j, j}, \sum_j u_{j,j}\rangle+ \sum_{i,j}\langle (-\Delta)^s\mu_0 u_{i, j}, u_{i, j}\rangle
+ \sum_{i,j}\langle (-\Delta)^s\mu_0 u_{i, j}, u_{j, i}\rangle$$
$$= \langle \sum_j (-\Delta)^s\lambda_0 u_{j, j}, \sum_j u_{j,j}\rangle+ \sum_{i,j}\langle (-\Delta)^s\mu_0 u_{i, j}, u_{i, j}\rangle
+ \sum_{i,j}\langle (-\Delta)^s\mu_0 u_{i, i}, u_{j, j}\rangle$$
$$= \langle \sum_j (-\Delta)^s(\lambda_0+ \mu_0)u_{j, j}, \sum_j u_{j,j}\rangle+ \sum_{i,j}\langle (-\Delta)^s\mu_0 u_{i, j}, u_{i, j}\rangle$$
$$\geq \sum_{i,j}\langle (-\Delta)^s\mu_0 u_{i, j}, u_{i, j}\rangle.$$
Then the coerciveness immediately follows from the $H^{1+s}$-norm equivalence
(see Subsection 2.2).

Now the Lax-Milgram theorem implies that the solution operator $f\to u_f$ associated with (\ref{LameOmega}) is well-defined, which is a homeomorphism
from $H^{-1-s}(\Omega)$ to $\tilde{H}^{s+1}(\Omega)$.

From now on we will always assume $0\leq\lambda, \mu\in L^\infty(\Omega)$, $\mu_0> 0$ and $\lambda_0+ \mu_0\geq 0$. 

Let $f:= -(-\Delta)^sL_{\lambda_0, \mu_0}g|_{\Omega}$ in (\ref{LameOmega}). Then we have the well-posedness of the exterior problem (\ref{fracLameext}).
\begin{prop}
For each $g\in H^{1+s}(\mathbb{R}^3)$, there exists a unique solution $u_g\in H^{1+s}(\mathbb{R}^3)$ of (\ref{fracLameext}) s.t. $u_g- g\in \tilde{H}^{1+s}(\Omega)$. Moreover, the solution operator $g\to u_g$
is bounded on $H^{1+s}(\mathbb{R}^3)$.
\end{prop}

\subsection{Dirichlet-to-Neumann map and integral identity}
Let $X:= H^{1+s}(\mathbb{R}^3)/\tilde{H}^{1+s}(\Omega)
= H^{1+s}(\Omega_e)$ 
and $\tilde{g}:=$ the natural image of 
$g\in H^{1+s}(\mathbb{R}^3)$ in $X$.

We define the Dirichlet-to-Neumann map $\Lambda$ by
$$\langle\Lambda \tilde{g}, \tilde{h}\rangle:= -B[u_g, h],\quad 
g,h\in H^{1+s}(\mathbb{R}^3)$$
where $u_g$ is the solution corresponding to the exterior data $g$ in (\ref{fracLameext}).

It is easy to verify that $\Lambda: X\to X^*$ is well-defined and this bilinear form definition coincides with the one given by (\ref{DN}) 
for $g\in C^\infty_c(\Omega_e)$.

For convenience, we will write 
$\Lambda g$ and $\langle \Lambda g, h\rangle$
instead of $\Lambda\tilde{g}$ and $\langle \Lambda\tilde{g}, \tilde{h}\rangle$.
Note that 
$$\langle\Lambda g, h\rangle= -B[u_g, u_h]= -B[u_h, u_g]
=\langle\Lambda h, g\rangle$$
so we have the integral identity
$$\langle\Lambda^{(1)} g^{(1)}, g^{(2)}\rangle- \langle\Lambda^{(2)} g^{(1)}, g^{(2)}\rangle
= -B^{(1)}[u^{(1)}, u^{(2)}]+ B^{(2)}[u^{(2)}, u^{(1)}]$$
\begin{equation}\label{intid}
=\int_{\Omega}(\lambda^{(2)}-\lambda^{(1)})(\sum_j u^{(1)}_{j,j})(\sum_j u^{(2)}_{j,j}) +\frac{1}{2}\int_\Omega(\mu^{(2)}-\mu^{(1)})\sum_{k,j}(u^{(1)}_{k, j}+ u^{(1)}_{j, k})
(u^{(2)}_{k, j}+ u^{(2)}_{j, k})
\end{equation}
where $\mathcal{L}^{(j)}$,$B^{(j)}$,$\lambda^{(j)}$ correspond to Lam\'e parameters $\lambda^{(j)}, \mu^{(j)}$; $u^{(j)}$ denotes the solution of
$$\mathcal{L}^{(j)}u= 0\,\,\,\, \text{in}\,\,\Omega,\qquad
u= g^{(j)}\,\,\,\, \text{in}\,\,\Omega_e.$$
We remark that this integral identity has the same form as its classical counterpart (see \cite{eskin2002inverse}).

\subsection{Unique continuation property and Runge approximation property}
Recall that a classical operator $L$ possesses the unique continuation property in a domain $U$ if 
$$\mathcal{L}u= 0\,\,\,\,\text{in}\,\,U,\quad
u= 0\,\,\,\,\text{in}\,\,V$$
where $V$ is a nonempty open subset of $U$ imply that $u= 0$ in $U$.

It is well-known that the classical constant coefficients Lam\'e operator $L_{\lambda_0, \mu_0}$ possesses the unique continuation property in this sense. (This property even holds true for the general variable coefficients Lam\'e operator $L_{\lambda, \mu}$. See for instance, the main theorem in \cite{ding1998unique}.)

The following proposition is the unique continuation property of $(-\Delta)^sL_{\lambda_0, \mu_0}$.
\begin{prop}\label{fracLUCP}
Let $u\in H^{1+s}(\mathbb{R}^3)$. Let $W$ be open. If $$(-\Delta)^sL_{\lambda_0, \mu_0}u= u= 0\quad\text{in}\,\,W,$$
then $u= 0$ in $\mathbb{R}^3$.
\end{prop}
\begin{proof}
By the unique continuation property of $(-\Delta)^s$ (Proposition \ref{UCP}), we have $L_{\lambda_0, \mu_0}\phi= 0$ in $\mathbb{R}^3$. Then by the unique continuation property of $L_{\lambda_0, \mu_0}$ we have $u= 0$ in $\mathbb{R}^3$.
\end{proof}

Based on the unique continuation property above, we can prove the following 
Runge approximation property.
\begin{prop}\label{RAP}
Let $W\subset \Omega_e$ be nonempty and open. Then
$$S:= \{u_g|_\Omega: g\in C^\infty_c(W)\}$$
is dense in $\tilde{H}^{1+s}(\Omega)$ where $u_g$ is the solution corresponding to the exterior data $g$ in (\ref{fracLameext}).
\end{prop}
\begin{proof}
By the Hahn-Banach Theorem, it suffices to show that:

If $f\in H^{-1-s}(\Omega)$ and $\langle f, w\rangle= 0$ for all $w\in S$, then $f= 0$. 

In fact, we can
choose $v\in \tilde{H}^{1+s}(\Omega)$ to be the solution of $\mathcal{L}v= f$ in $\Omega$. Then for any $g\in C^\infty_c(W)$, by the assumption $\langle f, w\rangle= 0$ for $w\in S$ we have
$$0= \langle f, u_g- g\rangle= -B[v, u_g- g]= B[v, g]$$
since $B[u_g, v]= 0$.
This implies $(-\Delta)^sL_{\lambda_0, \mu_0}v= 0$ in $W$. By the unique continuation property above we have $v= 0$ in $\mathbb{R}^3$. Hence $f= 0$.
\end{proof}

\subsection{Proof of Theorem 1.1}
Now we are ready to prove Theorem 1.1. The key point is to approximate certain carefully chosen functions by solutions of the linear exterior problems based on the Runge approximation property. This will enable us to exploit the integral identity (\ref{intid}) to determine $\lambda, \mu$.

\begin{proof}
Let $u^{(j)}$ denote the solution of
$$\mathcal{L}^{(j)}u= 0\,\,\,\, \text{in}\,\,\Omega,\qquad
u= g^{(j)}\,\,\,\, \text{in}\,\,\Omega_e.$$
For any given $\epsilon> 0$ and $f^{(j)}\in C^\infty_c(\Omega)$, by the Runge approximation property (Proposition \ref{RAP}), we can choose
$g^{(1)}\in C^\infty_c(W_1)$ s.t.
$$||u^{(1)}- f^{(1)}||_{\tilde{H}^{1+s}(\Omega)}\leq \epsilon$$
and for this chosen $g^{(1)}$, we can choose $g^{(2)}\in C^\infty_c(W_2)$ s.t.
$$||u^{(1)}||_{\tilde{H}^{1+s}(\Omega)}||u^{(2)}- f^{(2)}||_{\tilde{H}^{1+s}(\Omega)}\leq \epsilon.$$
(Actually we only need the $H^1$-norm approximation.) By the assumption 
$$\Lambda^{(1)}g|_{W_2}= \Lambda^{(2)}g|_{W_2}$$
for $g\in C^\infty_c(W_1)$ and the integral identity (\ref{intid}) we get
$$\int_{\Omega}(\lambda^{(2)}-\lambda^{(1)})(\sum_j u^{(1)}_{j,j})(\sum_j u^{(2)}_{j,j}) +\frac{1}{2}\int_\Omega(\mu^{(2)}-\mu^{(1)})\sum_{k,j}(u^{(1)}_{k, j}+ u^{(1)}_{j, k})(u^{(2)}_{k, j}+ u^{(2)}_{j, k})= 0.$$
Based on our choice for $g^{(j)}$, we get
$$|\int_{\Omega}(\lambda^{(2)}-\lambda^{(1)})(\sum_j f^{(1)}_{j,j})(\sum_j f^{(2)}_{j,j}) +\frac{1}{2}\int_\Omega(\mu^{(2)}-\mu^{(1)})\sum_{k,j}(f^{(1)}_{k, j}+ f^{(1)}_{j, k})
(f^{(2)}_{k, j}+ f^{(2)}_{j, k})|\leq C\epsilon$$
where $C$ is a constant depending on $\lambda^{(j)},\mu^{(j)}$ and $f^{(j)}$. Hence we conclude that
\begin{equation}\label{zerointid}
\int_{\Omega}(\lambda^{(2)}-\lambda^{(1)})(\sum_j f^{(1)}_{j,j})(\sum_j f^{(2)}_{j,j}) +\frac{1}{2}\int_\Omega(\mu^{(2)}-\mu^{(1)})\sum_{k,j}(f^{(1)}_{k, j}+ f^{(1)}_{j, k})
(f^{(2)}_{k, j}+ f^{(2)}_{j, k})= 0
\end{equation}
since $\epsilon$ is arbitrary.

We will appropriately choose $f^{(j)}$ in (\ref{zerointid}) to determine $\lambda, \mu$.

In fact, for any given $\psi\in C^\infty_c(\Omega)$,  we can choose $\phi$ s.t. $\phi\in C^\infty_c(\Omega)$ and 
$\phi= x_1$ on $\mathrm{supp}\, \psi$.

We can show that
$$\int_\Omega(\mu^{(2)}-\mu^{(1)})\partial_j \psi= 0,\quad (1\leq j\leq 3).$$
For instance, to obtain the equality above for $j=2$, we can choose 
$$f^{(1)}= (0, \phi, 0),\qquad f^{(2)}= (\psi, 0, 0)$$ 
in (\ref{zerointid}).
We can also show that
$$\int_\Omega(\lambda^{(2)}-\lambda^{(1)})\partial_j \psi= 0,
\quad (1\leq j\leq 3).$$
For instance, to obtain the equality above for $j=2$, we can choose 
$$f^{(1)}= (\phi,0,0),\qquad f^{(2)}= (0, \psi, 0)$$
in (\ref{zerointid}).
Hence we get 
$$\nabla(\mu^{(2)}-\mu^{(1)})= \nabla(\lambda^{(2)}-\lambda^{(1)})= 0.$$

Now we show that the constants $c_\mu:= \mu^{(2)}-\mu^{(1)}$ and 
$c_\lambda:= \lambda^{(2)}-\lambda^{(1)}$ are zeros. 

In fact, we can choose $\psi\in C^\infty_c(\Omega)$ s.t. $||\partial_1 \psi||_{L^2}\neq ||\partial_2 \psi||_{L^2}$.
Then we can show that
$$(c_\lambda+ 2c_\mu)||\partial_1 \psi||^2_{L^2}
+ c_\mu||\partial_2 \psi||^2_{L^2}+ c_\mu||\partial_3 \psi||^2_{L^2}= 0,$$
$$(c_\lambda+ 2c_\mu)||\partial_2 \psi||^2_{L^2}
+ c_\mu||\partial_3 \psi||^2_{L^2}+ c_\mu||\partial_1 \psi||^2_{L^2}= 0,$$
$$(c_\lambda+ 2c_\mu)||\partial_3 \psi||^2_{L^2}
+ c_\mu||\partial_1 \psi||^2_{L^2}+ c_\mu||\partial_2 \psi||^2_{L^2}= 0.$$
For instance, to obtain the first identity, we can choose 
$$f^{(1)}= f^{(2)}= (\psi, 0, 0)$$
in (\ref{zerointid}).
Now we combine the three identities to obtain $c_\lambda+ 4c_\mu= 0$. Then 
we combine the first two identities to obtain $c_\mu||\partial_1 \psi||^2_{L^2}= c_\mu||\partial_2 \psi||^2_{L^2}$, which implies 
$c_\mu= 0$ and thus $c_\lambda= 0$.
\end{proof}

\begin{remark}
Similar strategy has been applied to solve inverse problems for fractional Schr\"odinger operators with local perturbations. See \cite{cekic2020calderon, covi2020higher} for details.
\end{remark}

\subsection{Proof of Theorem 1.2}
We can identically apply the considerations for the exterior problem (\ref{fracLameext}) in previous subsections to the obstacle problem
(\ref{obeq}). (We just replace $\Omega$ by $\Omega\setminus\bar{D}$ in earlier arguments.) Our main task in this subsection is to prove the first part of Theorem 1.2. We will see that this part is an immediate consequence of the 
unique continuation property (Proposition \ref{fracLUCP}). Once we have determined the obstacle $D$, we can use the same arguments as in the proof of Theorem 1.1 to determine $\lambda, \mu$.

\begin{proof}
For the fixed nonzero $g\in C^\infty_c(W_1)$, let
$u^{(j)}$ denote the solution of the obstacle problem 
$$\mathcal{L}^{(j)}u= 0\,\, \text{in}\,\,\Omega\setminus\bar{D_j},\quad
u= 0\,\, \text{in}\,\,D_j,\quad
u= g\,\, \text{in}\,\,\Omega_e.$$
Since we have the assumption
$$\Lambda^{(1)}_{D_1} g|_{W_2}= \Lambda^{(2)}_{D_2} g|_{W_2}$$
and $u^{(1)}= u^{(2)}= g$ in $\Omega_e$, we get 
$$(-\Delta)^sL_{\lambda_0, \mu_0}(u^{(1)}- u^{(2)})= u^{(1)}- u^{(2)}= 0\quad\text{in}\,\,W_2.$$
Then Proposition \ref{fracLUCP} implies that $u^{(1)}= u^{(2)}$ in $\mathbb{R}^3$.

Suppose $D_1\neq D_2$. Without loss of generality
we can assume $V:= D_2\setminus\bar{D}_1$ is nonempty. Note that
$u^{(1)}= u^{(2)}= 0$ in $V$. Then the equation for $u^{(1)}$ in $V$ implies $(-\Delta)^sL_{\lambda_0, \mu_0}u^{(1)}= 0$ in $V$. But now Proposition \ref{fracLUCP} implies that
$u^{(1)}= 0$ in $\mathbb{R}^3$, which contradicts that $g$ is nonzero.
\end{proof}

\begin{remark}
Similar inverse obstacle problems have been studied for fractional elliptic operators. See \cite{cao2017simultaneously} for details.
\end{remark}

\section{Nonlinear fractional elasticity}
\subsection{Well-posedness and Dirichlet-to-Neumann map}
We first study the nonlinear equation
\begin{equation}\label{LNf}
\mathcal{L}u+ \mathcal{N}u= f\,\, \text{in}\,\,\Omega
\end{equation}
See (\ref{Nsum}) in Section 1 for the definition of $\mathcal{N}$.

From now on we will always assume $\mathcal{A}, \mathcal{B}, \mathcal{C}\in C^1(\bar\Omega)$ in (\ref{Nsum}). 
Let $\frac{1}{2}\leq s< 1$. Then $\frac{2n}{n-2s}\geq \frac{n}{2s}$ for $n= 3$ so we have the continuous embeddings (see Subsection 2.1)
$$\tilde{H}^s(\Omega)\hookrightarrow L^{\frac{2n}{n-2s}}(\Omega)
\hookrightarrow L^{\frac{n}{2s}}(\Omega)\hookrightarrow M(H^s\to H^{-s}).$$

Note that for $u\in \tilde{H}^{s+1}(\Omega)$, each component of $Nu$ is a sum of terms which have the form 
$\partial_j(a\phi\psi)$ where $a\in C^1(\bar\Omega)$ and $\phi,\psi\in \tilde{H}^{s}(\Omega)$ so $Nu\in H^{-1-s}(\mathbb{R}^3)$. 
Hence the map $F$ defined by
$$F(f, u):= (\mathcal{L}u+ \mathcal{N}u)|_\Omega- f$$
maps from $H^{-1-s}(\Omega)\times \tilde{H}^{s+1}(\Omega)$ to
$H^{-1-s}(\Omega)$.

Note that $F(0, 0)= 0$ and it is easy to verify that the Fr\'echet derivative
$$D_uF|_{(0,0)}(v)= \mathcal{L}v|_{\Omega},$$
which is a homeomorphism from $\tilde{H}^{s+1}(\Omega)$ to $H^{-1-s}(\Omega)$. By Implicit function theorem (see for instance, Theorem 10.6 in \cite{renardy2006introduction}), there exists $\delta> 0$ s.t. whenever
$||f||_{H^{-1-s}(\Omega)}\leq \delta$, we have both existence and uniqueness of small solutions of (\ref{LNf}), and $u_f$ smoothly depends on $f$.

Let $f:= -(-\Delta)^sL_{\lambda_0, \mu_0}g|_{\Omega}$ in (\ref{LNf}). Then we have the well-posedness of (\ref{NfracLameext}).
\begin{prop}
For each sufficiently small $g\in C^\infty_c(\Omega_e)$, there exists a unique small solution $u_g$ of the exterior problem (\ref{NfracLameext}) s.t. $u_g- g\in \tilde{H}^{1+s}(\Omega)$ and $u_g$ smoothly depends on $g$.
\end{prop}
Now we can conclude that the associated Dirichlet-to-Neumann map $\Lambda_N$ given by (\ref{DNforN}) is well-defined at least for small $g\in C^\infty_c(\Omega_e)$.

\subsection{Proof of Theorem 1.3}
We are ready to prove Theorem 1.3. We will first apply the first order linearization and the linear result (Theorem 1.1) to determine $\lambda, \mu$.
Then we will apply the second order linearization and the Runge approximation property (Proposition \ref{RAP}) to determine $\mathcal{A}, \mathcal{C}$.
\begin{proof}
\textbf{Determine $\lambda, \mu$}: Let $u^{(j)}_{\epsilon, g}$ be the solution of the exterior problem
\begin{equation}\label{LjNg}
\mathcal{L}^{(j)}u+ \mathcal{N}^{(j)}u= 0\,\,\,\,\text{in}\,\,\Omega,\qquad
u= \epsilon g\,\,\,\,\text{in}\,\,\Omega_e
\end{equation}
for $g\in C^\infty_c(W_1)$ and small $\epsilon$. 
Applying $\frac{\partial}{\partial\epsilon}|_{\epsilon= 0}$ to (\ref{LjNg}), we obtain that 
$$u^{(j)}_g:=\frac{\partial}{\partial\epsilon}|_{\epsilon= 0}u^{(j)}_{\epsilon, g}$$
is the solution of 
$$\mathcal{L}^{(j)}u= 0\,\,\,\,\text{in}\,\,\Omega,\qquad
u= g\,\,\,\,\text{in}\,\,\Omega_e.$$
Since we have the assumption
$$\Lambda^{(1)}_Ng= \Lambda^{(2)}_Ng\qquad \text{in}\,\,W_2,$$
$$i.e.\quad (-\Delta)^sL_{\lambda_0, \mu_0}u^{(1)}_{\epsilon, g}
= (-\Delta)^sL_{\lambda_0, \mu_0}u^{(2)}_{\epsilon, g}
\qquad \text{in}\,\,W_2,$$
we can apply $\frac{\partial}{\partial\epsilon}|_{\epsilon= 0}$ to the identity to obtain that
$$(-\Delta)^sL_{\lambda_0, \mu_0}u^{(1)}_g
= (-\Delta)^sL_{\lambda_0, \mu_0}u^{(2)}_g\qquad \text{in}\,\,W_2,$$
$$i.e.\quad \Lambda^{(1)}g= \Lambda^{(2)}g\qquad \text{in}\,\,W_2.$$
Hence we conclude that
$\lambda^{(1)}= \lambda^{(2)}:= \lambda$, $\mu^{(1)}= \mu^{(2)}:= \mu$
based on Theorem 1.1. 

Now we use $\mathcal{L}$ to denote both $\mathcal{L}^{(1)}$ and  $\mathcal{L}^{(2)}$.

\textbf{Determine $\mathcal{A}, \mathcal{C}$}:
Let $u^{(j)}_{\epsilon, g}$ be the solution of the exterior problem
\begin{equation}\label{LNjg}
\mathcal{L}u+ \mathcal{N}^{(j)}u= 0\,\,\,\,\text{in}\,\,\Omega,\quad
u= \epsilon_1g^{(1)}+ \epsilon_2g^{(2)}\,\,\,\,\text{in}\,\,\Omega_e
\end{equation}
for $g^{(j)}\in C^\infty_c(W_1)$ and small $\epsilon_j$.
First note that
$$\frac{\partial}{\partial\epsilon_j}|_{\epsilon_j= 0}u^{(1)}_{\epsilon, g}=\frac{\partial}{\partial\epsilon_j}|_{\epsilon_j= 0}u^{(2)}_{\epsilon, g}$$
since both of them are the solution of
$$\mathcal{L}u= 0\,\,\,\ \text{in}\,\,\Omega,\qquad
u= g^{(j)}\,\,\,\, \text{in}\,\,\Omega_e.$$
Hence we can denote both of them by $v^{(j)}$. 
Next we apply 
$\frac{\partial^2}{\partial\epsilon_1\partial\epsilon_2}|_{\epsilon_1= \epsilon_2= 0}$ to (\ref{LNjg}). Then $$w^{(j)}:=\frac{\partial^2}{\partial\epsilon_1\partial\epsilon_2}|_{\epsilon_1= \epsilon_2= 0}u^{(j)}_{\epsilon, g}$$
satisfies
\begin{equation}\label{Lw12}
\mathcal{L}w^{(j)}+ \tilde{N}^{(j)}(v^{(1)}, v^{(2)})= 0\,\,\,\, \text{in}\,\,\Omega,\qquad w^{(j)}= 0\,\,\,\, \text{in}\,\,\Omega_e.
\end{equation}
Here (based on (\ref{Nsum})) we can compute that
$$\tilde{N}^{(l)}_i= \sum^3_{j=1}\partial_j\tilde{N}^{(l)}_{ij},\quad 
(1\leq l\leq 2;\, 1\leq i\leq 3)$$
$$\tilde{N}^{(l)}_{ij}(v^{(1)}, v^{(2)})= (\lambda+ \mathcal{B})
\sum_{m,n}v^{(1)}_{m,n}v^{(2)}_{m,n}\delta_{ij}$$
$$+2\mathcal{C}^{(l)}(\sum_mv^{(1)}_{m,m})(\sum_mv^{(2)}_{m,m})\delta_{ij}
+ \mathcal{B}\sum_{m,n}v^{(1)}_{m,n}v^{(2)}_{m,n}\delta_{ij}
+ \mathcal{B}\sum_m(v^{(1)}_{m,m}v^{(2)}_{j,i}+ v^{(2)}_{m,m}v^{(1)}_{j,i})$$
$$+ \frac{\mathcal{A}^{(l)}}{4}\sum_m(v^{(1)}_{j,m}v^{(2)}_{m,i}+ v^{(2)}_{j,m}v^{(1)}_{m,i})+
(\lambda+ \mathcal{B})\sum_m(v^{(1)}_{m,m}v^{(2)}_{i,j}+ v^{(2)}_{m,m}v^{(1)}_{i,j})$$
$$+ (\mu+ \frac{\mathcal{A}^{(l)}}{4})
\sum_m(v^{(1)}_{m,i}v^{(2)}_{m,j}+ v^{(2)}_{m,i}v^{(1)}_{m,j}+
v^{(1)}_{i,m}v^{(2)}_{j,m}+ v^{(2)}_{i,m}v^{(1)}_{j,m}+
v^{(1)}_{i,m}v^{(2)}_{m,j}+ v^{(2)}_{i,m}v^{(1)}_{m,j}).$$

We also apply $\frac{\partial^2}{\partial\epsilon_1\partial\epsilon_2}|_{\epsilon_1= \epsilon_2= 0}$ to the Dirichlet-to-Neumann map assumption
$$(-\Delta)^sL_{\lambda_0, \mu_0}u^{(1)}_{\epsilon, g}
= (-\Delta)^sL_{\lambda_0, \mu_0}u^{(2)}_{\epsilon, g}
\qquad \text{in}\,\,W_2.$$
Then we get 
$$(-\Delta)^sL_{\lambda_0, \mu_0}w^{(1)}
= (-\Delta)^sL_{\lambda_0, \mu_0}w^{(2)}\qquad \text{in}\,\,W_2.$$
Since $w^{(1)}= w^{(2)}= 0$
in $\Omega_e$, the unique continuation property (Proposition \ref{fracLUCP}) implies that $w^{(1)}= w^{(2)}$ in $\mathbb{R}^3$.

Now we combine the two equations ($j= 1,2$) in (\ref{Lw12}) to obtain
\begin{equation}\label{Gij}
\sum^3_{j=1}\partial_j\tilde{G}_{ij}(v^{(1)},v^{(2)})
= 0,\quad (1\leq i\leq 3)
\end{equation}
where 
$$\tilde{G}_{ij}(v^{(1)},v^{(2)})= \tilde{N}^{(2)}_{ij}(v^{(1)}, v^{(2)})
-\tilde{N}^{(1)}_{ij}(v^{(1)}, v^{(2)})$$
$$=\frac{\mathcal{A}^{(2)}-\mathcal{A}^{(1)}}{4}(\sum_m (v^{(1)}_{j,m}+ v^{(1)}_{m,j})(v^{(2)}_{i,m}+ v^{(2)}_{m,i})+ \sum_m (v^{(2)}_{j,m}+ v^{(2)}_{m,j})(v^{(1)}_{i,m}+ v^{(1)}_{m,i}))$$
$$+2(\mathcal{C}^{(2)}-\mathcal{C}^{(1)})(\sum_mv^{(1)}_{m,m})
(\sum_mv^{(2)}_{m,m})\delta_{ij}.
$$
Let both sides of (\ref{Gij}) act on $\psi\in C^\infty_c(\Omega)$.
Then Proposition \ref{RAP} implies that
$$\sum^3_{j=1}\int\tilde{G}_{ij}(f^{(1)}, f^{(2)})\partial_j\psi= 0,\quad (1\leq i\leq 3)$$
for any $f^{(j)}\in C^\infty_c(\Omega)$.
We can appropriately choose $f^{(j)}$ to show that
$$\int(\mathcal{A}^{(2)}-\mathcal{A}^{(1)})\partial_j\psi= 0,
\quad (1\leq j\leq 3).$$
For instance, if we choose $i=1$,
$f^{(1)}= (0, \phi, 0)$ and $f^{(2)}= (\phi, 0, 0)$ where $\phi\in C^\infty_c(\Omega)$ satisfies $\phi= x_1$ on $\mathrm{supp}\,\psi$ for a chosen $\psi\in C^\infty_c(\Omega)$,  a direct computation shows that
$$\tilde{G}_{11}=0,\quad \tilde{G}_{12}=\frac{\mathcal{A}^{(2)}-\mathcal{A}^{(1)}}{2},\quad \tilde{G}_{13}=0,$$
which verifies the equality for $j= 2$.
Hence we conclude that $\mathcal{A}^{(2)}-\mathcal{A}^{(1)}:= c_{\mathcal{A}}$ is a constant.

We can also appropriately choose $f^{(j)}$ to show that
$$\int(\mathcal{C}^{(2)}-\mathcal{C}^{(1)})\partial_j\psi= 0,
\quad (1\leq j\leq 3).$$
For instance, if we choose $i=1$,
$f^{(1)}= (0, \varphi, 0)$ and $f^{(2)}= (\phi, 0, 0)$ where $\phi$ is defined as before and $\varphi\in C^\infty_c(\Omega)$ satisfies $\varphi= x_2$ on $\mathrm{supp}\,\psi$ for a chosen $\psi\in C^\infty_c(\Omega)$, a direct computation shows that
$$\tilde{G}_{11}= 2(\mathcal{C}^{(2)}-\mathcal{C}^{(1)}),\quad \tilde{G}_{12}=0,\quad \tilde{G}_{13}=0,$$
which verifies the equality for $j= 1$.
Hence we conclude that $\mathcal{C}^{(2)}-\mathcal{C}^{(1)}:= c_{\mathcal{C}}$ is a constant.

Now we show that $c_{\mathcal{A}}$ and $c_{\mathcal{C}}$ are zeros. 

In fact, we can choose $\psi\in C^\infty_c(\Omega)$ s.t. $||\partial_1 \psi||_{L^2}\neq ||\partial_2 \psi||_{L^2}$.

Then we can appropriately choose $f^{(j)}$ to obtain
$$(4c_{\mathcal{A}}+ 4c_{\mathcal{C}})||\partial_1 \psi||^2_{L^2}
+ c_{\mathcal{A}}||\partial_2 \psi||^2_{L^2}+ c_{\mathcal{A}}||\partial_3 \psi||^2_{L^2}= 0,$$
$$(4c_{\mathcal{A}}+ 4c_{\mathcal{C}})||\partial_2 \psi||^2_{L^2}
+ c_{\mathcal{A}}||\partial_3 \psi||^2_{L^2}+ c_{\mathcal{A}}||\partial_1 \psi||^2_{L^2}= 0,$$
$$(4c_{\mathcal{A}}+ 4c_{\mathcal{C}})||\partial_3 \psi||^2_{L^2}
+ c_{\mathcal{A}}||\partial_1 \psi||^2_{L^2}+ c_{\mathcal{A}}||\partial_2 \psi||^2_{L^2}= 0.$$
For instance, if we choose $i=1$,
$f^{(1)}= (\psi, 0, 0)$ and $f^{(2)}= (\phi, 0, 0)$ where $\phi$ is defined as before, a direct computation shows that
$$\tilde{G}_{11}= 2c_\mathcal{A}\partial_1\psi+ 2c_\mathcal{C}\partial_1\psi,\quad \tilde{G}_{12}= \frac{c_\mathcal{A}}{2}\partial_2\psi,\quad \tilde{G}_{13}= \frac{c_\mathcal{A}}{2}\partial_3\psi,$$
which verifies the first identity.
Now we combine the three identities to obtain $6c_\mathcal{A}+ 4c_\mathcal{C}= 0$. Then we combine the first two identities to obtain
$c_\mathcal{A}||\partial_1 \psi||^2_{L^2}= c_\mathcal{A}||\partial_2 \psi||^2_{L^2}$, which implies 
$c_\mathcal{A}= 0$ and thus $c_\mathcal{C}= 0$.

\end{proof}

\begin{remark}
The multiple-fold linearization procedure performed in the proof has been widely applied in solving inverse problems. For instance, see \cite{feizmohammadi2020inverse, krupchyk2020remark} for this approach for inverse problems for semilinear elliptic operators. Also see \cite{uhlmann2021inverse}
for this approach for an inverse problem for a nonlinear elastic wave operator.
\end{remark}

\bibliographystyle{plain}
{\small\bibliography{Reference5}}

\begin{thebibliography}{10}

\bibitem{bhattacharyya2021inverse}
Sombuddha Bhattacharyya, Tuhin Ghosh, and Gunther Uhlmann.
\newblock Inverse problems for the fractional-laplacian with lower order
  non-local perturbations.
\newblock {\em Transactions of the American Mathematical Society}, 2021.

\bibitem{bisci2016variational}
Giovanni~Molica Bisci, Vicentiu Radulescu, and Raffaella Servadei.
\newblock {\em Variational methods for nonlocal fractional problems}, volume
  162.
\newblock Cambridge University Press, 2016.

\bibitem{cao2017simultaneously}
Xinlin Cao, Yi-Hsuan Lin, and Hongyu Liu.
\newblock Simultaneously recovering potentials and embedded obstacles for
  anisotropic fractional {Schr\"odinger} operators.
\newblock {\em Inverse Problems and Imaging}, 13:197--210, 2019.

\bibitem{cekic2020calderon}
Mihajlo Ceki{\'c}, Yi-Hsuan Lin, and Angkana R{\"u}land.
\newblock The {Calder\'on} problem for the fractional {Schr\"odinger} equation
  with drift.
\newblock {\em Calculus of Variations and Partial Differential Equations},
  59(91), 2020.

\bibitem{covi2020inverse}
Giovanni Covi.
\newblock An inverse problem for the fractional {Schr\"odinger} equation in a
  magnetic field.
\newblock {\em Inverse Problems}, 36(4):045004, 2020.

\bibitem{covi2020higher}
Giovanni Covi, Keijo M{\"o}nkk{\"o}nen, Jesse Railo, and Gunther Uhlmann.
\newblock The higher order fractional {Calder\'on} problem for linear local
  operators: uniqueness.
\newblock {\em arXiv preprint arXiv:2008.10227}, 2020.

\bibitem{de2020nonlinear}
Maarten De~Hoop, Gunther Uhlmann, and Yiran Wang.
\newblock Nonlinear interaction of waves in elastodynamics and an inverse
  problem.
\newblock {\em Mathematische Annalen}, 376(1):765--795, 2020.

\bibitem{ding1998unique}
Dang Ding~Ang, Dang Duc Trong~And, Masaru Ikehata, and Masahiro Yamamoto.
\newblock Unique continuation for a stationary isotropic {Lam\'e} system with
  variable coefficients.
\newblock {\em Communications in partial differential equations},
  23(1-2):599--617, 1998.

\bibitem{eskin2002inverse}
Gregory Eskin and James Ralston.
\newblock On the inverse boundary value problem for linear isotropic
  elasticity.
\newblock {\em Inverse Problems}, 18(3):907--921, 2002.

\bibitem{feizmohammadi2020inverse}
Ali Feizmohammadi and Lauri Oksanen.
\newblock An inverse problem for a semi-linear elliptic equation in riemannian
  geometries.
\newblock {\em Journal of Differential Equations}, 269(6):4683--4719, 2020.

\bibitem{ghosh2017calderon}
Tuhin Ghosh, Yi-Hsuan Lin, and Jingni Xiao.
\newblock The {Calder\'on} problem for variable coefficients nonlocal elliptic
  operators.
\newblock {\em Communications in Partial Differential Equations},
  42(12):1923--1961, 2017.

\bibitem{ghosh2020uniqueness}
Tuhin Ghosh, Angkana R{\"u}land, Mikko Salo, and Gunther Uhlmann.
\newblock Uniqueness and reconstruction for the fractional {Calder\'on} problem
  with a single measurement.
\newblock {\em Journal of Functional Analysis}, page 108505, 2020.

\bibitem{ghosh2020calderon}
Tuhin Ghosh, Mikko Salo, and Gunther Uhlmann.
\newblock The {Calder\'on} problem for the fractional {Schr\"odinger} equation.
\newblock {\em Analysis \& PDE}, 13(2):455--475, 2020.

\bibitem{krupchyk2020remark}
Katya Krupchyk and Gunther Uhlmann.
\newblock A remark on partial data inverse problems for semilinear elliptic
  equations.
\newblock {\em Proceedings of the American Mathematical Society},
  148(2):681--685, 2020.

\bibitem{lai2020calderon}
Ru-Yu Lai, Yi-Hsuan Lin, and Angkana R{\"u}land.
\newblock The {Calder\'on} problem for a space-time fractional parabolic
  equation.
\newblock {\em SIAM Journal on Mathematical Analysis}, 52(3):2655--2688, 2020.

\bibitem{lai2021inverse}
Ru-Yu Lai and Ting Zhou.
\newblock An inverse problem for non-linear fractional magnetic schrodinger
  equation.
\newblock {\em arXiv preprint arXiv:2103.08180}, 2021.

\bibitem{li2020determining}
Li~Li.
\newblock Determining the magnetic potential in the fractional magnetic
  {Calder\'on} problem.
\newblock {\em Communications in Partial Differential Equations, in press},
  2020.

\bibitem{li2021fractional}
Li~Li.
\newblock A fractional parabolic inverse problem involving a time-dependent
  magnetic potential.
\newblock {\em SIAM Journal on Mathematical Analysis}, 53(1):435--452, 2021.

\bibitem{li2021inverse}
Li~Li.
\newblock An inverse problem for a fractional diffusion equation with
  fractional power type nonlinearities.
\newblock {\em arXiv preprint arXiv:2104.00132}, 2021.

\bibitem{nakamura1994global}
Gen Nakamura and Gunther Uhlmann.
\newblock Global uniqueness for an inverse boundary problem arising in
  elasticity.
\newblock {\em Inventiones mathematicae}, 118(1):457--474, 1994.

\bibitem{nakamura2003global}
Gen Nakamura and Gunther Uhlmann.
\newblock Global uniqueness for an inverse boundary value problem arising in
  elasticity ({Erratum}).
\newblock {\em Inventiones mathematicae}, 152(1):205--207, 2003.

\bibitem{renardy2006introduction}
Michael Renardy and Robert Rogers.
\newblock {\em An introduction to partial differential equations}, volume~13.
\newblock Springer Science \& Business Media, 2006.

\bibitem{ruland2020fractional}
Angkana R{\"u}land and Mikko Salo.
\newblock The fractional {Calder\'on} problem: low regularity and stability.
\newblock {\em Nonlinear Analysis}, 193:111529, 2020.

\bibitem{tarasov2019fractional}
Vasily Tarasov and Elias Aifantis.
\newblock On fractional and fractal formulations of gradient linear and
  nonlinear elasticity.
\newblock {\em Acta Mechanica}, 230(6):2043--2070, 2019.

\bibitem{uhlmann2021inverse}
Gunther Uhlmann and Jian Zhai.
\newblock On an inverse boundary value problem for a nonlinear elastic wave
  equation.
\newblock {\em Journal de Math{\'e}matiques Pures et Appliqu{\'e}es},
  153:114--136, 2021.

\end{thebibliography}
\end{document}